\documentclass[english,12pt,oneside]{amsproc}
\usepackage[english]{babel}
\usepackage{a4wide}
\usepackage{amsthm}
\usepackage{graphics}
\usepackage{amsfonts, amssymb, amscd, amsmath}
\usepackage{mathdots}
\usepackage{latexsym}
\usepackage[matrix,arrow,curve]{xy}
\usepackage{mathabx}
\usepackage{color}
\usepackage{pbox}
\usepackage{tikz}
\usetikzlibrary{matrix,decorations.pathreplacing,positioning}
\usepackage{hyperref}

\DeclareMathOperator{\Ker}{Ker}

\DeclareMathOperator{\Hom}{Hom}
\DeclareMathOperator{\rk}{rk}
\DeclareMathOperator{\odd}{odd}

\DeclareMathOperator{\com}{compl}

\DeclareMathOperator{\ord}{ord}

\DeclareMathOperator{\str}{star}
\DeclareMathOperator{\IC}{IC}
\DeclareMathOperator{\Flats}{Flats}
\DeclareMathOperator{\Cl}{Cl}
\DeclareMathOperator{\drk}{drk}
\DeclareMathOperator{\Gr}{Gr}
\DeclareMathOperator{\Fl}{Fl}

\newcommand{\ttt}{\mathfrak{t}}
\newcommand{\hhh}{\mathfrak{h}}

\newcounter{stmcounter}[section]

\numberwithin{equation}{section}

\newcounter{thmMaincounter}

\theoremstyle{plain}

\newtheorem{thmM}[thmMaincounter]{Theorem}

\newtheorem{prop}[stmcounter]{Proposition}
\newtheorem{lem}[stmcounter]{Lemma}

\newtheorem{algor}[stmcounter]{Algorithm}

\theoremstyle{definition}
\newtheorem{defin}[stmcounter]{Definition}

\theoremstyle{remark}
\newtheorem{ex}[stmcounter]{Example}
\newtheorem{rem}[stmcounter]{Remark}
\newtheorem{con}[stmcounter]{Construction}

\newcommand{\CP}{\mathbb{C}P}

\newcommand{\A}{\mathcal{A}}

\newcommand{\Pp}{\mathbb{P}}
\newcommand{\Stg}{S_{tg}}

\newcommand{\G}{\mathcal{G}}
\newcommand{\Hh}{\mathcal{H}}

\def\Co{\mathbb C}
\def\Ro{\mathbb R}
\def\Qo{\mathbb Q}
\def\Zo{\mathbb Z}

\begin{document}
\title[Matroids in toric topology]{Matroids in toric topology}

\author{Anton Ayzenberg}
\address{Faculty of computer science, National Research University Higher School of Economics, Russian Federation}
\email{ayzenberga@gmail.com}

\author{Vladislav Cherepanov}
\address{Faculty of computer science, National Research University Higher School of Economics, Russian Federation; Current address: Mathematical Institute, University of Oxford, Oxford OX2 6GG, United Kingdom}
\email{vilamsenton@gmail.com, vladislav.cherepanov@queens.ox.ac.uk}

\thanks{The article was prepared within the framework of the HSE University Basic Research Program}

\keywords{Torus action, invariant submanifold, homology of posets, matroids, geometric lattices, GKM-theory}

\subjclass[2020]{Primary: 57S12, 55N91, 13F55, 06A06 Secondary: 55U10, 57R91, 13H10}

\begin{abstract}
In this paper we study general torus actions on manifolds with isolated fixed points from combinatorial point of view. The main object of study is the poset of face submanifolds of such actions. We introduce the notion of a locally geometric poset --- the graded poset locally modelled by geometric lattices, and prove that for any torus action, the poset of its faces is locally geometric. Next we discuss the relations between posets of faces and GKM-theory. In particular, we define the face poset of an abstract GKM-graph and show how to reconstruct the face poset of a manifold from its GKM-graph.
\end{abstract}

\maketitle


\section{Introduction}\label{secIntro}

Toric topology studies actions of a compact torus $T^k$ on closed smooth manifolds $X^{2n}$ in terms of related combinatorial structures. The classical examples are given by smooth toric varieties~\cite{CLSch} (which are classified by their simplicial fans), and their topological analogues --- quasitoric manifolds~\cite{BPnew} (classified by characteristic pairs). In both cases the half--dimensional torus $T^n$ acts on a manifold $X^{2n}$ with $H^{\odd}(X^{2n})=0$, and it happens that the poset $S(X^{2n})$ of $T^n$-invariant submanifolds in $X^{2n}$ is a certain cell subdivision of a topological disc, see~\cite{MasPan}. The poset $S(X^{2n})$ has nice acyclicity properties: not only $S(X^{2n})$ is acyclic (which is obvious since it has the greatest element), but its skeleta and the links of its simplices also are. These acyclicity properties are intimately related to the Cohen--Macaulay property of face algebras. Face algebra is isomorphic to equivariant cohomology of $X$ (see~\cite{BPnew} for details of all constructions).

In this paper we study combinatorial structures related to more general torus actions. For an action of $T=T^k$ on $X=X^{2n}$ with isolated fixed points, we study combinatorial structure of the poset $S(X)$ of face submanifolds of $X$.

We introduce the following definition.

\begin{defin}
A finite poset $S$ is called locally geometric of rank $k$ if the following holds
\begin{enumerate}
  \item $S$ is graded. All minimal elements have rank $0$.
  \item $S$ has the greatest element $\hat{1}$ with $\rk \hat{1}=k$.
  \item For any element $s\in S$, the upper order ideal $S_{\geqslant s}$ is a geometric lattice of rank $k-\rk s$.
\end{enumerate}
\end{defin}

Our main result is the following

\begin{thmM}\label{thmMainUpperAcycl}
Consider an effective action of $T=T^k$ on a manifold $X=X^{2n}$ with $0<\#X^T<\infty$, and let $S(X)$ be the poset of face submanifolds of $X$. Then $S(X)$ is locally geometric of rank $k$. For a fixed point $x\in X^T$, the poset $S(X)_{\geqslant x}$ is isomorphic to the lattice of flats of the linear matroid determined by the set of weights $\alpha_{x,1},\ldots,\alpha_{x,n}\in\Hom(T,S^1)\otimes\Qo$ of the tangent representation $T_xX$.
\end{thmM}

It follows that for any face $s\in S_X$, the poset $S(X)_{>s}\setminus\{\hat{1}\}$ is shellable (according to Bj\"{o}rner~\cite{Bjorner}), and therefore its geometrical realization is homotopy equivalent to a wedge of $(k-\rk s-2)$-dimensional spheres. In particular, $S(X)_{>s}\setminus\{\hat{1}\}$ is $(k-\rk S-3)$-acyclic.

Besides Theorem~\ref{thmMainUpperAcycl}, we note that the face poset $S(X)$ of any torus action carries an additional combinatorial structure, namely the dimensions of faces. The existence of this structure puts additional restrictions on the combinatorics of face posets. In Section~\ref{secCoherentPosets} we introduce the notion of coherent locally geometric posets and prove that face posets of torus actions are coherent.

We then proceed to define the face posets for an abstract GKM-graph $\G$ which is assumed connected. A face of a GKM-graph can be defined in two possible ways: either as a GKM-subgraph or as a totally geodesic GKM-subgraph. This gives rise to two posets $S(\G)$, and $\Stg(\G)$. Unlike Theorem~\ref{thmMainUpperAcycl}, these posets are not necessarily locally geometric. However, there exists a relation between the face poset $S(X)$ of a (weakly) GKM-manifold and the posets of the corresponding GKM-graph.

\begin{thmM}\label{thmCore}
Let $X$ be a GKM-manifold. Then there exist Galois insertions
\[
S(X)\hookrightarrow S(\G(X)) \mbox{ and } S(X)\hookrightarrow \Stg(\G(X)).
\]
\end{thmM}

It should be noted that the study of matroids (finite geometries) with relation to torus actions also appeared in the seminal work of Gelfand--Serganova~\cite{GelfSerg}, where the more general notion of a Coxeter matroid was introduced. However, the constructions of our paper are different from those of~\cite{GelfSerg}. The focus of~\cite{GelfSerg} is made on algebraic torus actions on homogeneous spaces, and matroids were used to describe the types of orbit closures in the Grassmann manifolds, all matroids are representable over $\Co$. On the other hand, in our paper, geometric lattices arise as the face posets of linear representations of compact tori, matroids are represantable over $\Qo$. There may be a connection between the two theories, but it is not straightforward.

The paper has the following structure. In Section~\ref{secFaces}, we strictly define the notion of a face of a torus action (there are certain issues with non-connected stabilizers where one should be careful in the definitions). We prove Theorem~\ref{thmMainUpperAcycl} in Section~\ref{secPosets} and introduce the dimension coherence condition in Section~\ref{secCoherentPosets}. Next we turn to the setting relevant to GKM theory. In Section~\ref{secGKM}, we recall the combinatorial basics of GKM-theory, as well as the notion of Galois connection.

\section{Construction of the face poset}\label{secFaces}

In this section we define the faces of an action and list their main properties. Let a torus $T=T^k$ act effectively on a topological space $X$ which is assumed connected. Let $\Cl(T)$ denote the set of all closed subgroups of $T$. For a point $x\in X$, $T_x\in \Cl(T)$ denotes the stabilizer (the stationary subgroup) of $x$, and $Tx\subset X$ is the orbit of $x$. In the following we assume that $X$ is a $T$-CW-complex (see~\cite[Def.1.1]{AdDav}). In particular, this holds for smooth torus actions on smooth manifolds. 


\begin{con}\label{conExactPartition}
For an action of $T$ on $X$ we define the \emph{fine subdivision} of $X$ by stabilizer types:
\[
X=\bigsqcup_{H\in \Cl(T)}X^{(H)}
\]
where $X^{(H)}=\tilde{\lambda}^{-1}(H)=\{x\in T\mid T_x=H\}$. Moreover, for $H\in \Cl(T)$ define
\[
X^H=\bigsqcup_{\tilde{H}\supseteq H}X^{(\tilde{H})}=\left\{x\in X\mid hx=x\,\,\forall h\in H\right\}.
\]
Therefore $X^H$ is the set of $H$-fixed points of $X$. 
\end{con}

\begin{con}\label{conFiltrationOrbitType}
For a $T$-action on a topological space $X$ consider the filtration
\begin{equation}\label{eqEquivFiltrGeneralX}
X_0\subset X_1\subset\cdots \subset X_k
\end{equation}
where $X_i$ is the union of all $\leqslant i$-dimensional orbits of the action. In other words,
\[
X_i=\{x\in X\mid \dim T_x\geqslant k-i\}=\bigsqcup_{H\in \Cl(T),\dim H\geqslant k-i}X^{(H)}
\]
according to the natural homeomorphism $Tx\cong T/T_x$. Filtration~\eqref{eqEquivFiltrGeneralX} is called the orbit type filtration, and $X_i$ the equivariant $i$-skeleton of $X$. Each $X_i$ is $T$-stable. Orbit type filtration induces the following filtration of the orbit space $Q=X/T$:
\begin{equation}\label{eqEquivFiltrGeneralQ}
Q_0\subset Q_1\subset\cdots \subset Q_k,\quad\mbox{where }Q_i=X_i/T
\end{equation}
\end{con}

Recall that, for a (compact) torus $T$, the lattice $N=\Hom(T,S^1)\cong \Zo^k$ is called the weight lattice, and its dual lattice $N^*=\Hom(S^1,T)$ is called the lattice of 1-dimensional subgroups.

\begin{con}\label{conTangentAndRankDim}
Let $X$ be a smooth closed connected orientable manifold, and let $T$ act smoothly and effectively on $X$. If $x\in X^T$ is a fixed point, we have an induced representation of $T$ in the tangent space $T_xX$ called the tangent representation. Let $\alpha_{x,1},\ldots,\alpha_{x,n}\in \Hom(T^k,S^1)\cong \Zo^{k}$ be the weights of the tangent representation at $x$, which means  that
\[
T_xX\cong V(\alpha_{x,1})\oplus\cdots\oplus V(\alpha_{x,n})\oplus \Ro^{\dim X-2n}
\]
where $V(\alpha)$ is the standard 1-dimensional complex representation given by $tz=\alpha(t)\cdot z$, $z\in \Co$, and the action on $\Ro^{\dim X-2n}$ is trivial (see~\cite[Cor.I.2.1]{Hsiang}). It is assumed that all weight vectors $\alpha_{x,i}$ are nonzero since otherwise the corresponding summands contribute to $\Ro^{\dim X-2n}$. If there is no $T$-invariant complex structure on $X$, then there is an ambiguity in the choice of signs of vectors $\alpha_i$. However, for the statements of this paper the choice of signs is nonessential. We can also assume that the weight vectors $\alpha_{x,1},\ldots,\alpha_{x,n}$ linearly span the weight lattice $\Hom(T^k,S^1)$ (otherwise there would exist an element $\lambda$ of the dual lattice $\Hom(S^1,T^k)$ such that $\langle\alpha_{x,1},\lambda\rangle=0$, which implies that the corresponding 1-dimensional subgroup $\lambda$ lies in the noneffective kernel). This observation implies that, if the action has fixed points, there holds $\dim X\geqslant 2n\geqslant 2k$.
\end{con}

\begin{defin}\label{definComplexity}
Let $T$ act effectively on a smooth manifold $X$, and the fixed point set $X^T$ is finite and nonempty. The nonnegative integer $\com X=\frac{1}{2}\dim X-\dim T$ is called \emph{the complexity of the action}.
\end{defin}

If the action is noneffective, the symbol $\com X$ denotes the complexity of the corresponding effective action, i.e. the action of the quotient by the noneffective kernel.

\begin{prop}[{Consequence of slice theorem, see e.g.~\cite[Ch.II, Thm.4.2 and Ch.IV]{Bred}}]\label{propSliceThm}
Each fixed point $x\in X^T$ has a neighborhood equivariantly diffeomorphic to the tangent representation $T_xX$.
\end{prop}

In particular, $x\in X^T$ is isolated if and only if $\dim X=2n$ and all tangent weights $\alpha_{x,1},\ldots,\alpha_{x,n}$ are nonzero.

For each closed subgroup $H\subset T$, the subset $X^H$ is a closed smooth submanifold in $X$ (this is again a consequence of general slice theorem, see~\cite[Ch.IV]{Bred}). This submanifold is stable under $T$ as follows from commutativity of $T$.

\begin{defin}\label{definInvarSubmfds}
A connected component of $X^H$ is called an \emph{invariant submanifold}.
\end{defin}

For a smooth $T$-action on $X$ consider the canonical projection $p\colon X\to Q$ to the orbit space, and the filtration~\eqref{eqEquivFiltrGeneralQ} of the orbit space.

\begin{defin}\label{definFacesOrbitSpaceGeneral}
The closure of a connected component of $Q_i\setminus Q_{i-1}$ is called a \emph{face} $F$ if it contains at least one fixed point. The number $i$ is called the rank of a face $F$, it is equal to the dimension of a generic $T$-orbit in $F$.
\end{defin}

%

\begin{lem}\label{lemFaceSubmanifold}
For a $T$-action on $X$, the full preimage $X_F=p^{-1}(F)$ of any face $F\subset Q$ is a smooth submanifold in $X$.
\end{lem}

\begin{proof}
Let $F$ be the closure of a connected component $F^{\circ}$ of $Q_i\setminus Q_{i-1}$. Let $H$ be the noneffective kernel of the $T$-action on $X_F$. This implies $X_F$ lies inside the closed submanifold $X^H$. According to the principal orbit theorem \cite[Thm.3.1]{Bred}, there exists a subgroup $H'$, $\dim H'=\dim H$, such that the subset $U\subset X_F$ of orbits with stabilizer $H'$ is open, dense, and connected in a connected component of the manifold $X^H$. Therefore the set $U/T\subset Q_i\setminus Q_{i-1}$ is connected and contains $F^{\circ}$. This implies that the closure of $U/T$ coincides with $F$, and we have that $X_F$ is a connected component of $X^H$, hence a smooth submanifold.
\end{proof}

\begin{defin}\label{definFaceSubmfd}
Let $F$ be a face of $Q=X/T$. The submanifold $X_F=p^{-1}(F)\subset X$ is called a \emph{face submanifold} corresponding to $F$.
\end{defin}

By the definition of a face, each face contains at least one fixed point of the $T$-action on $X$. On the other hand, there exist sufficiently many faces in the vicinity of each fixed point, as the following lemma shows.

\begin{lem}\label{lemFaceNearFixedPoint}
Let $y$ be a point in a small neighborhood of $x\in X^T$ (or, equivalently, $y\in T_xX$). Let $H$ be the connected component of the stabilizer $T_y$, containing $1$. Then the connected component $Y$ of $X^H$ containing $x$ is a face submanifold.
\end{lem}

\begin{proof}
First notice that $X^T\subset X^H$, so $Y$ is a nonempty closed smooth submanifold. If $i=\dim Ty$, then $p(y)\in X_i\setminus X_{i-1}$. The orbit $p(y)$ lies in some connected component $F$ of $Q_i\setminus Q_{i-1}$. Since $y$ is close to $x$, $F$ contains $x$, so $F$ is a face, and $X_F$ is a face submanifold. Similarly to the proof of Lemma~\ref{lemFaceSubmanifold}, we apply principal orbit theorem: there exists a connected open dense subset $U$ of both $X_F$ and some component of $X^{H'}$, and all orbits of $U$ have the same stabilizer $H'$, $\dim H'=\dim H$. Since $H$ is connected, there is an inclusion $H\subseteq H'$, so we have $X^{H'}\subseteq X^H$. Therefore, $U$ is an open subset in both manifolds $X_F$ and $Y$, so the latter two coincide.
\end{proof}

\begin{ex}\label{exNonFaceFace}
Here we provide an example of an action, which has invariant submanifolds that are not face submanifolds. Let us consider locally standard actions of $T^n$ on $X^{2n}$. Locally standard means that the orbit space is a nice $n$-manifold with corners $Q^n=X^{2n}/T^n$ with the projection $p\colon X^{2n}\to Q^n$. In this setting, each face $F$ of $Q$ (face of a mfd. with corners) gives an invariant submanifold $p^{-1}(F)$. However, $F$ is a face in the sense of Definition~\ref{definFacesOrbitSpaceGeneral} if and only if it contains a vertex of $Q$.

There certainly exist manifolds with corners with a face that does not contain a vertex. An example is shown on Fig.~\ref{figEye}: the inner circle is a face of a manifold with corners, but it does not contain a fixed point of the torus action, so it is not a face in the sense of Definition~\ref{definFacesOrbitSpaceGeneral}. There is a way to construct a locally standard action on a 4-manifold, whose orbit space is the one shown on the figure. Consider two manifolds:
\begin{itemize}
  \item $X_1=S^3\times S^1=\{(z_1,z_2,z_3)\in\Co^3\mid |z_1|^2+|z_2|^2=1, |z_3|=1\}$ with the locally standard action of $T^2$ given by $(t_1,t_2)(z_1,z_2,z_3)=(z_1,t_1z_2,t_2z_3)$. The orbit space of this action is homeomorphic to $D^2$.
  \item $X_2=S^4$ with the canonical action of $T^2$ (the one-point compactification of the standard action of $T^2$ on $\Co^2$). The orbit space is a biangle.
\end{itemize}
The connected sum $X=X_1\hash_{T^2}X_2$ of these manifolds along a 2-dimensional orbit is a smooth 4-manifold with a locally standard action, and its orbit space is the connected sum of $X_1/T^2$ and $X_2/T^2$ at an interior point, which is diffeomorphic to the manifold with corners shown on Fig.~\ref{figEye}.
\end{ex}

\begin{figure}[h]
\begin{center}
\includegraphics[scale=0.4]{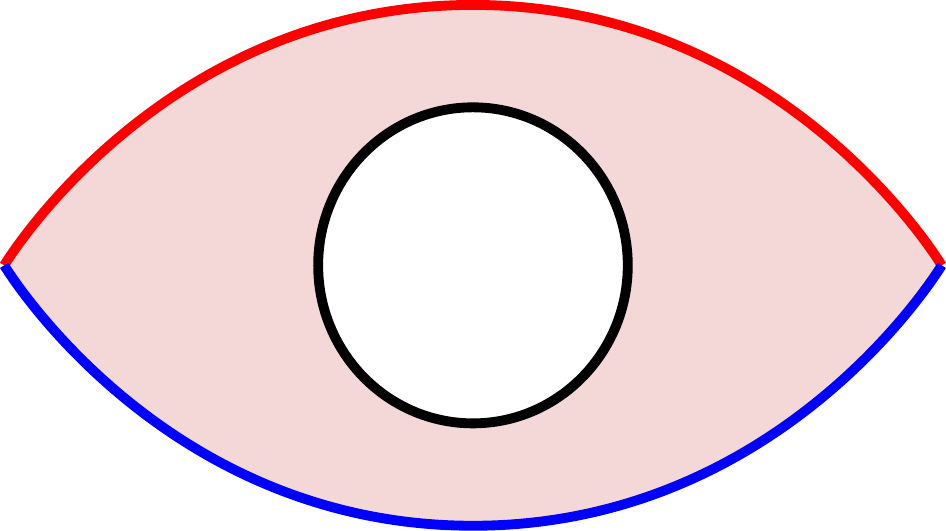}
\end{center}
\caption{The interior circle in this manifold with corners is not considered a face according to Definition~\ref{definFacesOrbitSpaceGeneral}.}\label{figEye}
\end{figure}

\begin{con}\label{conPosetOfFaces}
It is known~\cite[Thm.5.11]{Dieck} that a smooth action of $T$ on a compact smooth manifold has only finite number of stabilizers. Therefore, there is only finite number of invariant (hence, face) submanifolds of $X$. The finite set of faces of $Q$ (equiv., face submanifolds of $X$) is partially ordered by inclusion. This poset is denoted by $S(X)$.
\end{con}

\begin{con}\label{conCommonStabilizer}
Let us introduce additional notation. Let $F$ be a face and $X_F$ the corresponding face submanifold. The action of $T$ on $X_F$ may have a noneffective kernel
\[
T_F=\{t\in T\mid tx=x\,\,\forall x\in X_F\},
\]
which we also call a common stabilizer of points from $F$. The number $\dim T/T_F$ is called the rank of $F$ and is denoted $\rk F$. Ranks define the grading on the poset $S(X)$ (this will become more transparent in the proof of Lemma~\ref{lemUpperToLocal} below).

The effective action of $T/T_F$ on $X_F$ satisfies the general assumption from Definition~\ref{definComplexity}: its fixed point set is nonempty (due to the definition of a face) and finite (because so is $X^T$). Therefore, the induced complexity $\com X_F$ is well defined: $\com X_F = \dim X_F - \dim(T/T_F) = \dim X_F - \rk F$.
\end{con}

Notice that each face $F$ of rank $r$ is a subset of the filtration term $Q_r$. It is tempting to say that $Q_r$ is the union of all faces of rank $r$. However, this may be false in general, since faces are required to contain a fixed point by the definition. Some invariant submanifolds of rank $r$ contributing to $Q_r$ may fail to be face submanifolds.

\section{Local structure of face posets}\label{secPosets}

\subsection{Local structure of $S(X)$}

Recall that $S(X)$ denotes the poset of face submanifolds in $X$ (or the poset of faces of $Q$) ordered by inclusion.

The symbol $|S|$ denotes the geometrical realization of a poset $S$ that is the geometrical realization of the order complex $\ord S$ (the simplices of $\ord S$ are all chains in $S$).

Let $x\in X^T$ be a fixed point and $V=T_xX$ be the tangent representation. As was noted earlier, we have a decomposition into irreducible representations $V=V(\alpha_{x,1})\oplus\ldots\oplus V(\alpha_{x,n})$, where $\alpha = \{\alpha_{x,1},\ldots,\alpha_{x,n}\} \subset \Hom(T^k,T^1)\cong \Zo^k$ is the set of tangent weights. Notice that the lattice $\Hom(T^k,T^1)$ can be naturally embedded into the dual Lie algebra $\ttt^*\cong \Ro^k$ of the torus $T=T^k$.

The next lemma explains that the local structure of the poset $S(X)$ near a fixed point $x$ is given by a linearization of the torus action at that point. We slightly abuse the notation by not distinguishing a fixed point $x$ from the set $\{x\}$ considered as an element of $S(X)$: we simply write $x$ in both cases.

\begin{lem}\label{lemUpperToLocal}
For a fixed point $x\in X^T$, the subposet $S(X)_{\geqslant x}$ is isomorphic to the face poset $S(V)$ of the $T^k$-action on $V=T_xX\cong \Ro^{2n}$.
\end{lem}

\begin{proof}
Here we essentially use slice theorem (Proposition~\ref{propSliceThm}) and Lemma~\ref{lemFaceNearFixedPoint}. Define a map of posets $\pi\colon S(X)_{\geqslant x}\to S(V)$ as follows. Let $Y\in S(X)$ be a face submanifold such that $x\in Y$, and $Y$ is a connected component of $X^H$ for some connected subgroup $H\in \Cl(T)$. The tangent space $T_xY$ coincides with $(T_xX)^H$. We set $\pi(Y)=T_xY$, this is an element of $S(V)$.

Similarly, a map $\rho\colon S(V)\to S(X)_{\geqslant x}$ can be defined. For $Z\in S(V)$ we pick up a connected subgroup $H\in\Cl(T)$ such that $Z=V^H$ and set $\rho(Z)$ to be the component of $X^H$ containing $x$.

The maps $\pi$ and $\rho$ are inverses of each other according to the slice theorem.
\end{proof}

In order to prove Theorem~\ref{thmMainUpperAcycl}, it only remains to prove that $S(V)$ is a geometric lattice for any effective representation of $T$ on $V\cong\Ro^{2n}$.

\subsection{Linear matroids and geometric lattices}

We recall several basic constructions from matroid theory.

\begin{con}\label{conLinearMatroids}
A collection $\beta=(\beta_1,\ldots,\beta_m)$ of vectors in a vector space $W$ (over any field), is called a linear matroid. An independence complex of a matroid $\beta$ is the simplicial complex
\[
\IC(\beta)=\{\{i_1,\ldots,i_k\}\subset[m]\mid \beta_{i_1},\ldots,\beta_{i_k} \mbox{ are linearly independent}\}.
\]
A flat is an intersection of $\beta$ with any vector subspace $\Pi\subset W$. More precisely, a flat is any subset $A\subseteq[m]$ for which there exists a vector subspace $\Pi\subseteq W$ such that
\[
A=\{i\in [m]\mid \beta_i\in \Pi\}.
\]
The flats are naturally ordered by inclusion, and the poset of flats is denoted $\Flats(\beta)$. Each flat comes equipped with its rank: the rank of a flat $A$ is given by the dimension of the linear span $\dim \langle\beta_i\mid i\in A\rangle$. This makes  $\Flats(\beta)$ a graded poset. Notice that $\Flats(\beta)$ has the greatest element $\hat{1}$ (given by $[m]$ itself) and the least element $\hat{0}$ (given by the set of indices of zero vectors in the collection $\beta$). The rank of a matroid $\beta$ is by definition the dimension of the linear span of all $\beta_i's$.
\end{con}

\begin{defin}
The poset $\Flats(\beta)$ is called \emph{the geometric lattice} (or the lattice of flats) of the linear matroid $\beta$.
\end{defin}

There exists a more general abstract notion of a matroid and its geometric lattice, which we do not need here, see details in~\cite{Oxley}. We recall the classical results of Bj\"{o}rner.

\begin{prop}[{\cite[Thms.7.3.3, 7.6.3, 7.8.1, 7.9.1]{Bjorner}}]\label{propMatroidsAcycl}
(1) The simplicial complex $\IC(\beta)$ is a shellable complex, hence it is Cohen--Macaulay. The geometrical realization $|\IC(\beta)|$ is homotopy equivalent to a wedge of $(\rk\beta-1)$-dimensional spheres. The number of spheres in the wedge is given by the top $h$-number $h_{\rk\beta-1}(\IC(\beta))$.

(2) The poset $\Flats(\beta)\setminus\{\hat{0},\hat{1}\}$ is shellable. The geometrical realization $|\Flats(\beta)\setminus\{\hat{0},\hat{1}\}|$ is homotopy equivalent to a wedge of $(\rk\beta-2)$-dimensional spheres. The number of spheres in the wedge is given by the M\"{o}bius function $\mu_{\Flats(\beta)}(\hat{0},\hat{1})$ of the poset $\Flats(\beta)$.
\end{prop}

\subsection{Poset of invariant subrepresentations}

We recall that
\[
V=V(\alpha_{1})\oplus\ldots\oplus V(\alpha_{n})
\]
is a decomposition of a $T$-representation into real 2-dimensional irreducible representations.

\begin{prop}\label{propLocalLattice}
The poset $S(V)$ of invariant subrepresentations is naturally isomorphic to the geometric lattice of the linear matroid $\alpha=(\alpha_1,\ldots,\alpha_n)\subset \Hom(T,S^1)\otimes\Qo$ of weights of the representation.
\end{prop}

\begin{proof}
Let $H\in S(T)$ be an arbitrary subgroup. Consider
\[
V^H=\left\{v=(v_1,\ldots,v_n)\in V=\bigoplus\nolimits_{i=1}^nV(\alpha_i)\mid Hv=v\right\}.
\]
Since $V$ splits in the sum of irreducible representations, we can treat each ``coordinate'' $v_i\in V(\alpha_i)$ separately. Let $\hhh\subset\ttt$ be the tangent algebra of the subgroup $H\subset T$. The equality $Hv_i=v_i$ leads to the following alternative:
\begin{enumerate}
  \item either $\alpha_i(\hhh)=0$ (equiv., $\alpha_i\in \hhh^\bot$) and $v_i$ is any vector from $V(\alpha_i)$;
  \item or $\alpha_i(\hhh)\neq 0$ and necessarily $v_i=0$.
\end{enumerate}
This alternative implies the following decomposition
\[
V^H=\bigoplus_{\alpha_i\in \hhh^\bot}V(\alpha_i).
\]

So far, the invariant subrepresentation (which is automatically a face submanifold, since it contains the fixed point $0$) is a certain coordinate subspace of $V$. This coordinate subspace is encoded by the collection of indices $A_H=\{i\in [m]\mid\alpha_i\in \hhh^\bot\}$. Since $A_H$ is determined by the intersection of $\alpha$ with the vector subspace $\hhh^\bot$, the subset $A_H$ is a flat: we have $A_H\in \Flats(\alpha)$.

Now we consider the map of posets $\psi\colon \Flats(\alpha)\to S(V)$ defined as follows. Let $A=\{\alpha_i\}\cap\Pi\in \Flats(\alpha)^*$ be a flat, where $\Pi$ is the linear span of $\{\alpha_i\mid i\in A\}$. Since $\alpha_i$'s are integral, the vector subspace $\Pi\subseteq\ttt^*$ is rational. Hence its orthogonal $\Pi^\bot$ is a tangent algebra $\hhh$ to some closed subgroup $H\subset T$. We set $\psi(A)=V^H=\bigoplus_{\alpha_i\in \hhh^\bot}V(\alpha_i)$. The bigger is $A$, the smaller is $H$, hence the bigger is the invariant subrepresentation $V^H$. Therefore $\psi$ is the monotonic map of posets. It is injective by construction and surjective by the arguments in the first part of the proof.
\end{proof}

Lemma~\ref{lemUpperToLocal} and Proposition~\ref{propLocalLattice} finish the proof of Theorem~\ref{thmMainUpperAcycl}. Indeed, if $s\in S(X)$ is any element of $S(X)$, then there exists a fixed point $x\in X^T$ such that $s\geqslant x$ (by the definition of face submanifolds). Therefore $S(X)_{\geqslant s}$ is a subposet in $S(X)_{\geqslant x}$ consisting of elements greater then the given element $s$. Proposition~\ref{propLocalLattice} and Lemma~\ref{lemUpperToLocal} imply that $S(X)_{\geqslant x}$ is a geometric lattice $\Flats(\alpha)$. It is well-known that the subposet $\Flats(\alpha)_{\geqslant t}$ is again a geometric lattice, for any element $t\in \Flats(\alpha)$.

\subsection{Examples and constructions}

Recall from the introduction that a poset $S$ is called locally geometric if for any $s\in S$, the subposet $S_{\geqslant s}$ is a geometric lattice. In a much similar fashion, we can define dually locally geometric posets as the posets whose lower order ideals $S_{\leqslant s}$ are geometric lattices. Equivalently, $S$ is dually locally geometric if and only if $S^*$ (the poset with reversed order) is locally geometric.

According to Theorem~\ref{thmMainUpperAcycl}, many examples of locally geometric posets are provided by the face posets of torus actions.

\begin{ex}
We have proved in Proposition~\ref{propLocalLattice}, that a linear $T$-representation $V$ gives a poset $S(V)$ which is the lattice of flats for the matroid of weights of the representation. Certainly, every geometric lattice is a locally geometric poset. If a geometric lattice (or the corresponding matroid) is not representable over $\Qo$, then it is not isomorphic to the face poset of any torus action.
\end{ex}

\begin{ex}
Consider a locally standard action of $T^n$ on $X^{2n}$, i.e. the action locally modeled by the standard action $T^n$ on $\Co^n$. At each fixed point $x\in X^{2n}$ the weights $\alpha_{x,1},\ldots,\alpha_{x,n}$ form a basis of $\Hom(T^n,S^1)\otimes\Qo$. The corresponding geometric lattice is therefore isomorphic to the Boolean lattice $B_n$. Hence the face poset $S(X)$ has the property that for any $s$ the subposet $S(X)_{\geqslant s}$ is a Boolean lattice. In other words, $S(X)^*$ is a simplicial poset.

The relation between locally standard actions and simplicial complexes (and, more generally, simplicial posets) is certainly well-known and well-studied in toric topology, see~\cite{BPnew}.
\end{ex}

The next few examples show how to construct new locally geometric posets from a given geometric lattice.

\begin{ex}\label{exCompactify}\emph{ Compactification.}
Consider the following construction. Let $S$ be a geometric lattice and $\hat{0}$ its bottom element. Let us double the bottom element: $\overline{S}=S\sqcup\{\hat{0}'\}$, where $\hat{0}'<s$ for any $s\neq\hat{0}$, and $\hat{0}'$, $\hat{0}$ are incomparable. It is easily seen that $\overline{S}$ is a locally geometric poset.

Posets of this form are combinatorial counterparts of the Alexandroff compactifications of linear representations. Let $V\cong\Ro^{2n}$ be a linear representation of a torus $T$. There is an induced $T$-action on the Alexandroff's one-point compactification $\overline{V}=V\sqcup\{\infty\}\cong S^{2n}$. Each (linear) face submanifold $F\subset V$ except $F=\{0\}$ gives a face submanifold $F\sqcup\{\infty\}$ of $\overline{V}$. On the other hand, $\infty$ is a fixed point of the action on $\overline{V}$ contained in all proper face submanifolds --- it corresponds to $\hat{0}'$. Hence there is a natural isomorphism
\[
S(\overline{V})\cong \overline{S(V)}.
\]
\end{ex}

\begin{ex}\label{exProjectivize} \emph{Projectivization.}
Let $S$ be a geometric lattice of rank $k$. It is easily seen that the poset $\Pp(S)=S\setminus\hat{0}$ is a locally geometric poset of rank $k-1$. We will call $\Pp(S)$ a projectivization of $S$. The name is explained similarly to the previous example.

Let $V$ be a linear $T$-representation. With some choice of signs of the weights we get an isomorphism $V=V(\alpha_1)\oplus\cdots\oplus V(\alpha_n)\cong\Co^n$, in other words, we get $T$-invariant complex structure on $V$. Then we are able to take the projectivization $\Pp(V)\cong\CP^{n-1}$, the smooth variety of all complex lines in $V$. The linear action on $V$ induces the action of $T$ on $\Pp(V)$. The face submanifolds of $\Pp(V)$ are all vector face submanifolds of $V$ except $\{0\}$. Hence there is a natural isomorphism
\[
S(\Pp(V))=\Pp(S(V)).
\]
Since the induced action of $T$ on $\Pp(V)$ has diagonal torus as its non-effective kernel $\Delta(T^1)\subset T$, so the rank of the corresponding effective action is reduced by~1.
\end{ex}

\begin{rem}
Example~\eqref{exProjectivize} suggests that there might exist more general constructions: combinatorial analogues of grassmanization and flagization. However, these constructions should be highly nontrivial. Even if one starts with the standard representation $\Co^n$ of a torus $T^n$ (whose face poset is the boolean lattice), the posets of faces of the Grassmann manifold $\Gr_{n,k}$ and full flag variety $\Fl_n$ are quite complicated.
\end{rem}

\section{Dimension coherence condition}\label{secCoherentPosets}

By definition, any locally geometric poset carries the rank function. The rank function is compatible with the rank function on a geometric lattice: for any minimal element $x\in S$, and for any element $s\geqslant x$ the rank of $s$ in $S$ is equal to the rank of $s$ in the geometric lattice $S_{\geqslant x}$. For the face posets $S(X)$ of torus actions the rank function $\rk F$ indicates the effective dimension of a torus action on a face $F$. The importance of the rank function follows from the fact that it induces the filtration~\eqref{eqEquivFiltrGeneralX} of $X$ by the equivariant skeleta. This filtration is widely used in the analysis of equivariant cohomology, see for example~\cite{FP,Franz}.

\begin{con}\label{conDimFunctionTopology}
There is another important data attached to $S(X)$ --- the topological dimensions of faces. For each face $F\in S(X)$, the corresponding face submanifold $X_F$ has even dimension (over $\Ro$). One can thus introduce a function $\drk\colon S(X)\to\Zo_{\geqslant0}$ by
\[
\drk(F)=\frac{1}{2}\dim X_F \mbox{ for any } F\in S(X).
\]
The arguments from Construction~\ref{conTangentAndRankDim} imply that $\drk(F)\geqslant\rk(F)$. Moreover, $\drk$ is strictly monotonic function on $S(X)$: if $F\subset F'$, then $\drk(F)<\drk(F')$. We can then define the face-complexity function
\[
\com(F)=\drk(F)-\rk(F)\geqslant 0.
\]
It is not difficult to prove (and it was proved in~\cite[Lm.3.1]{AyzMasEquiv}), that $\com$ is non-strictly monotonic on $S(X)$ as well.
\end{con}

The following observations are quite straightforward, however they can be used to determine the function $\drk$ in many cases without a priori knowledge of the topology of~$X$. We recall that the elements of rank $1$ of a geometric lattice are called atoms.

\begin{lem}\label{lemSXisCoherent}
Let $S(X)$ be the face poset of a manifold $X$ with a torus action having isolated fixed points. Then, for any minimal element $x\in S(X)$ (i.e. a fixed point) one has
\[
\dim X=2\drk(\hat{1})=\sum_{a\in S(X), a>x, \rk a=1}2\drk(X_a).
\]
The summation is taken over all atoms of $S(X)_{\geqslant x}$.
\end{lem}

\begin{proof}
As before, we look at the tangent representation $T_xX$ and the geometric lattice $S(T_xX)\cong S(X)_{\geqslant x}$. The atoms of $S(T_xX)$ are the 1-dimensional flats in the linear matroid determined by the tangent weights $\alpha_1,\ldots,\alpha_n$ at $x$. Therefore, any atom $a$ is spanned by some weight, and the corresponding dimension-value $\drk(a)$ equals the number of weights collinear to $a$. Therefore the summands in the decomposition of $T_xX$ into irreducibles can be collected into groups as follows:
\[
T_xX\cong \bigoplus_{i=1}^nV(\alpha_i)=\bigoplus_{a\in S(T_xX), \rk a=1}\bigoplus_{\alpha_i\in a}V(\alpha_i).
\]
It is easily seen that $2\drk(a)=\dim \bigoplus_{\alpha_i\in a}V(\alpha_i)=\#\{i\mid\alpha_i\in a\}$ and $\dim X=2\drk(\hat{1})=2n$, which proves the desired formula.
\end{proof}

\begin{rem}\label{remCoherentRemark}
Notice that the statement of Lemma~\ref{lemSXisCoherent} applies not only to $\hat{1}\in S(X)$, but to any element of $S(X)$ in a similar way. Indeed, if $F\in S(X)$ is a face, then $S(X)_{\leqslant F}$ is isomorphic to $S(X_F)$ for the corresponding face submanifold $X_F$. Therefore, for any fixed point $x\leqslant F$, we have
\[
\drk(F)=\sum_{x<a\leqslant F, \rk a=1}\drk(X_a).
\]
The summation over empty set is assumed zero.
\end{rem}

Lemma~\ref{lemSXisCoherent} and Remark~\ref{remCoherentRemark} suggest the following definition.

\begin{defin}\label{definCoherentPosets}
Let $S$ be a locally geometric poset, and $\A$ denote the set of its atoms, $\A=\{a\in S\mid \rk a=1\}$. Let $d\colon\A\to\Zo_{>0}$ be a function valued in positive integers. The pair $(S,d)$ is called a coherent locally geometric poset if, for any $s\in S$, the number
\begin{equation}\label{eqDRKnumberAxiom}
\sum\nolimits_{a\in\A, x<a\leqslant s}d(a)
\end{equation}
does not depend on the choice of the minimal element $x\leqslant s$, $\rk x=0$. The number given by~\eqref{eqDRKnumberAxiom} is denoted $\drk(s)$. The monotonic function $\drk\colon S\to\Zo_{\geqslant 0}$ defined this way will be called the dimension-rank of a coherent locally geometric poset.
\end{defin}

In terms of this definition, Lemma~\ref{lemSXisCoherent} and Remark~\ref{remCoherentRemark} imply that the face poset $S(X)$ of any torus action is coherent, and the function $\drk$ tracks the dimensions of face submanifolds.

To restrict ourselves to a reasonable yet interesting class of examples we give the following definition.

\begin{defin}
A locally geometric poset $S$ is called GKM-coherent, if $(S,d_{GKM})$ is coherent for a function $d_{GKM}$ taking value $1$ at each atom of $S$.
\end{defin}

\begin{rem}\label{remGKMtheoryCoherent}
This definition is motivated by the GKM-theory, see Section~\ref{secGKM}. Condition 3 in Definition~\ref{defGKMmfd} implies $d(a)=1$. This observation implies that the poset $S(X)$ of a GKM-manifold is necessarily GKM-coherent. 
\end{rem}

\begin{ex}\label{exStupidUnion}
There exist examples of locally geometric posets, which are not GKM-coherent. To construct an example, we take two geometric lattices $S_1,S_2$ of equal ranks, but having different numbers of atoms. We then take their disjoint union, and identify their top elements:
\[
S=S_1\sqcup S_2/(\hat{1}_{S_1}\sim \hat{1}_{S_2}),
\]
making $\hat{1}=[\hat{1}_{S_1}]=[\hat{1}_{S_2}]$ the top element. The resulting poset $S$ is locally geometric by construction. However, the number $\drk(\hat{1})=\sum\nolimits_{a\in\A, x<a}d(a)$ depends on whether $x=\hat{0}_{S_1}$ or $x=\hat{0}_{S_2}$.
\end{ex}

\section{Faces in GKM-graphs vs. faces in GKM-manifolds}\label{secGKM}

\subsection{GKM-theory}

\begin{defin}\label{definIndependent}
A $T$ action on $X$ is called $j$-independent, if, for any fixed point $x\in X^T$, any subset of cardinality $\leqslant j$ of tangent weights $\alpha_{x,1},\ldots,\alpha_{x,n}$ is linearly independent.
\end{defin}

In some previous works we also used the term ``action in $j$-general position'' instead of $j$-independent.

Let us recall the basics of GKM-theory. The details can be found in the original paper of Goresky--Kottwitz--MacPherson~\cite{GKM}, or in a review~\cite{Kur}, which fits our topological setting better.

\begin{defin}\label{defGKMmfd}
A $2n$-dimensional (orientable connected) compact manifold $X$ with an action of $T=T^k$ is called \emph{a GKM-manifold}, if the following conditions hold:
\begin{enumerate}
\item $X$ is equivariantly formal;
\item The fixed point set $X_0=X^T$ is finite and nonempty;
\item Each face submanifold of rank $1$ has dimension $2$.
\end{enumerate}
We call an action weakly GKM, if it only satisfies conditions 2 and 3.
\end{defin}

\begin{rem}\label{rem2indepIsWeakGKM}
The weak GKM-action condition is equivalent to $2$-independence. Indeed, if there exist two collinear weights they fall in the same 1-flat of the matroid $\Flats(\{\alpha_{x,i}\})$, and therefore the corresponding face submanifold of $X$ would have dimension at least $4$. On the other hand, if an action is $2$-independent, each tangent weight $\alpha_{x,1}$ spans a 1-dimensional flat in $\Flats(\alpha_{x,i})\cong S(X)_{\geqslant x}$, hence the corresponding element of $S(X)$ has dimension 2.
\end{rem}

In a weak GKM manifold, every face submanifold of rank $1$ is a 2-dimensional orientable surface with an action of a circle. Hence it is a 2-sphere with exactly 2 fixed points. The union of all face submanifolds of rank $\leqslant 1$ is therefore a union of $2$-spheres, intersecting in common fixed points. The 1-skeleton $S(X)_1$ is a graph (without loops, but it might have multiple edges). Each edge $e$ of this graph (i.e. atom of $S(X)$) comes equipped with the weight $\alpha(e)\in\Hom(T,S^1)$, that is the weight of the tangent representation to the 2-sphere $X_e$ at any of its two fixed points. This gives rise to the definition of a GKM-graph. This is similar to the definition of GKM-graph given in~\cite{BGH}, however we avoid mentioning connection at the moment.

\begin{defin}\label{defGKMgeneral}
An abstract GKM-graph $\G$ is a finite $n$-regular connected graph $(V,E)$ (possibly with multiple edges) equipped with a function $\alpha\colon E\to \Hom(T^k,S^1)\cong\Zo^k$, which satisfies
\begin{enumerate}
  \item $\alpha(xy)=\pm\alpha(yx)$ for all edges $e=(xy)$,
  \item for each edges $e_1=(xy)$ and $e_2=(yz)$, the values $\alpha(e_1),\alpha(e_2)$ are linearly independent in $\Hom(T^k,S^1)\otimes\Qo$, and there exists an edge $e_3=(zu)$ such that $\alpha(e_3)$ lies in the rational span $\langle \alpha(e_1),\alpha(e_2)\rangle$.
\end{enumerate}
The function $\alpha$ is called \emph{an axial function}. In the signed version of a GKM-graph, it is required that $\alpha$ is defined on oriented edges, and changes sign if the orientation is reversed: $\alpha(xy)=-\alpha(yx)$.
\end{defin}

\begin{con}\label{conRankInGKMgraph}
For a vertex $x$ of a graph $\G$, let $\str(x)$ denote the set of edges incident to $x$ (assumed outgoing in the signed case).
The rational span 
\begin{equation}\label{eqFlatOfGraph}
\Pi_{\G}=\langle\alpha(e)\mid e\in \str(x)\rangle\in\Hom(T,S^1)\otimes\Qo
\end{equation} 
does not depend on a vertex $v$. Indeed, item (2) of the definition implies that this span coincides for adjacent vertices, and the rest follows from connectivity of $\G$. The dimension of this vector space is called the rank of the GKM-graph $\G$. Usually, it coincides with $k$, --- this situation corresponds to effective torus actions. The number $n=|\str(x)|$ for any $x$ is called the dimension of $\G$.
\end{con}

\begin{defin}\label{defGKMconnection}
A GKM-graph $\G$ with connection is a signed GKM-graph equipped with additional data, the connection defined as follows. \emph{A connection} $\theta$, is a collection of bijections $\theta_{(xy)}\colon \str(x)\to \str(y)$ for all edges $(xy)$ of a graph, satisfying the properties:
\begin{enumerate}
 \item $\theta_{e}e=e$ for all edges $e$;
 \item $\theta_{(yx)}^{-1}=\theta_{(xy)}$;
 \item The integral vector $\alpha(\theta_{(xy)}e)-\alpha(e)$ is collinear to $\alpha(xy)$ for any $e\in\str(x)$.
\end{enumerate}
\end{defin}

The following statement is standard (see e.g. \cite[Thm.3.4]{BGH} for item 3), but we review its proof in the terms of this paper.

\begin{prop}\label{propGKMmfdToGrph}
The following assertions hold true
\begin{enumerate}
  \item If a $T$-action on $X$ is $2$-independent, then $\G(X)$ is an abstract GKM-graph.
  \item If a $T$-action is $j$-independent, then for any $x\in X^T$, any $l\leqslant j-1$ tangent weights span a unique face submanifold of $X$.
  \item If an action is $3$-independent, then there is a canonical connection on its GKM-graph.
\end{enumerate}
\end{prop}

\begin{proof}
(1) The axial function is constructed from the weights of torus actions on the corresponding 2-spheres. We only need to check item 2 of Definition~\ref{defGKMgeneral}. Let $x,y,z,u$, and $e_1,e_2$ be as in the definition. There exists face $F$ of rank 2 which corresponds to the 2-flat $\langle \alpha(e_1),\alpha(e_2)\rangle\in \Hom(T,S^1)\otimes\Qo$ via the isomorphism $S_{\geqslant y}\cong \Flats(\alpha_{z,i})$. On the other hand, $F$ lies in $S_{\geqslant z}\cong\Flats(\alpha_{z,i})$, hence there should exist at least one weight at $z$ which spans the same flat together with $\alpha(e_2)$.

(2) Here we need to prove that, for any $l\leqslant j-1$ tangent weights $\alpha_{x,1},\ldots,\alpha_{x,l}$, there exists a unique face submanifold $X_F$ of rank $l$ in $X$ with exactly these tangent weights. Since the action is $j$-independent, any $l\leqslant j-1$ tangent vectors at $x$ form a flat. This flat can be considered an element of $S(X)$ via the isomorphisms $S(X)_{\geqslant x}\cong S(TX)\cong \Flats(\{\alpha_{x,i}\})$ proved earlier.

(3) The last statement follows easily. Indeed, if the action is $3$-independent, then any $2$ tangent weights span a face submanifold $X_F$ of type $(2,2)$. Its GKM-graph is 2-regular, hence it is a union of cycle graphs. In a cycle graph, a connection can be defined uniquely which proves the statement.
\end{proof}

If $X$ is a GKM-manifold and $\G(X)$ is its GKM-graph, their notions of dimension and rank are consistent (of course, in the sense that real dimension of $X$ is twice the dimension of $\G(X)$). Notice that, whenever $X_F$ is a face submanifold of a weak GKM-manifold $X$, it is weak GKM as well and $\G(X_F)$ is a subgraph of $\G(X)$ equipped with the same axial function.

\begin{rem}\label{remNonConnected}
The proof of the last item of Proposition~\ref{propGKMmfdToGrph} may seem strange at a first glance: why do we take the union of cycle graphs if it is usually just a single cycle graph? Unfortunately, some anomalous situations may occur when a connected weak GKM manifold has disconnected GKM-graph. To get examples, take the connected sum of two GKM manifolds of the same type along their free orbits (the idea has much in common with Example~\ref{exNonFaceFace}). We want to avoid such anomalous situations, so we give the following definition.
\end{rem}

\begin{defin}\label{definNormalGKM}
A weak GKM manifold $X$ is called normal if any face submanifold of $X$ has connected GKM-graph.
\end{defin}

\begin{rem}
Notice that any GKM manifold (i.e. weakly GKM + equivariantly formal) is normal. In equivariantly formal case any face submanifold is formal as well. The $1$-skeleton of $2$-independent equivariantly formal action is connected. This fact is well known in the Hamiltonian case (see~\cite{GZ}) and in all known examples, and in general it is proved in the forthcoming paper~\cite{AyzMasSolo} (this statement appears implicitly in~\cite{AyzMasEquiv}).
\end{rem}

\subsection{Faces of GKM-graphs}

It is a standard approach in toric topology and toric geometry to analyze GKM-graphs instead of GKM-manifolds. Therefore it seems reasonable to study face posets defined directly from GKM-graphs. In this section we realize this idea.

\begin{defin}[{\cite[Def.1.4.2]{GZ}}]
Let $\G$ be an abstract GKM-graph with an axial function $\alpha$. A connected subgraph $\Hh\subset \G$ is called a face of rank $j$ if it is a GKM-graph of rank $j$. In the case $G$ has connection, a face $\Hh$ is called totally geodesic, if for any edge $xy\in \Hh$ there holds $\theta_{xy}(\str(x)\cap \Hh)=\str(y)\cap \Hh$.
\end{defin}

Let $S(\G)$ denote the poset of all faces of $\G$ ordered by inclusion, and $\Stg(\G)$ the poset of totally geodesic faces (if $\G$ has connection).

Unlike Theorem~\ref{thmMainUpperAcycl}, the posets $S(\G)$ and $\Stg(\G)$ need not be geometric posets.

\begin{ex}
Consider the GKM-graph $\G_6$ of the full flag manifold $\Fl_3$, shown in the top left corner of Fig.~\ref{figFlags}. It is assumed that the weight of each edge coincides with its direction on the plane $\Ro^2$ where the graph is drawn, and the connection $\theta_{yz}$ maps an edge $xy$ to a unique edge $zu$ which is not parallel to $xy$ on the plane. The faces of rank 0 and 1 are vertices and edges of the graph respectively. The rest of Fig.~\ref{figFlags} shows all faces of rank 2. All totally geodesic faces are shown in the top row. There are too many of them for the posets $S(\G_6)_{\geqslant x}$ and $\Stg(\G_6)_{\geqslant x}$ to be geometric lattices of rank 2.
\end{ex}

\begin{figure}[h]
\begin{center}
\includegraphics[scale=0.35]{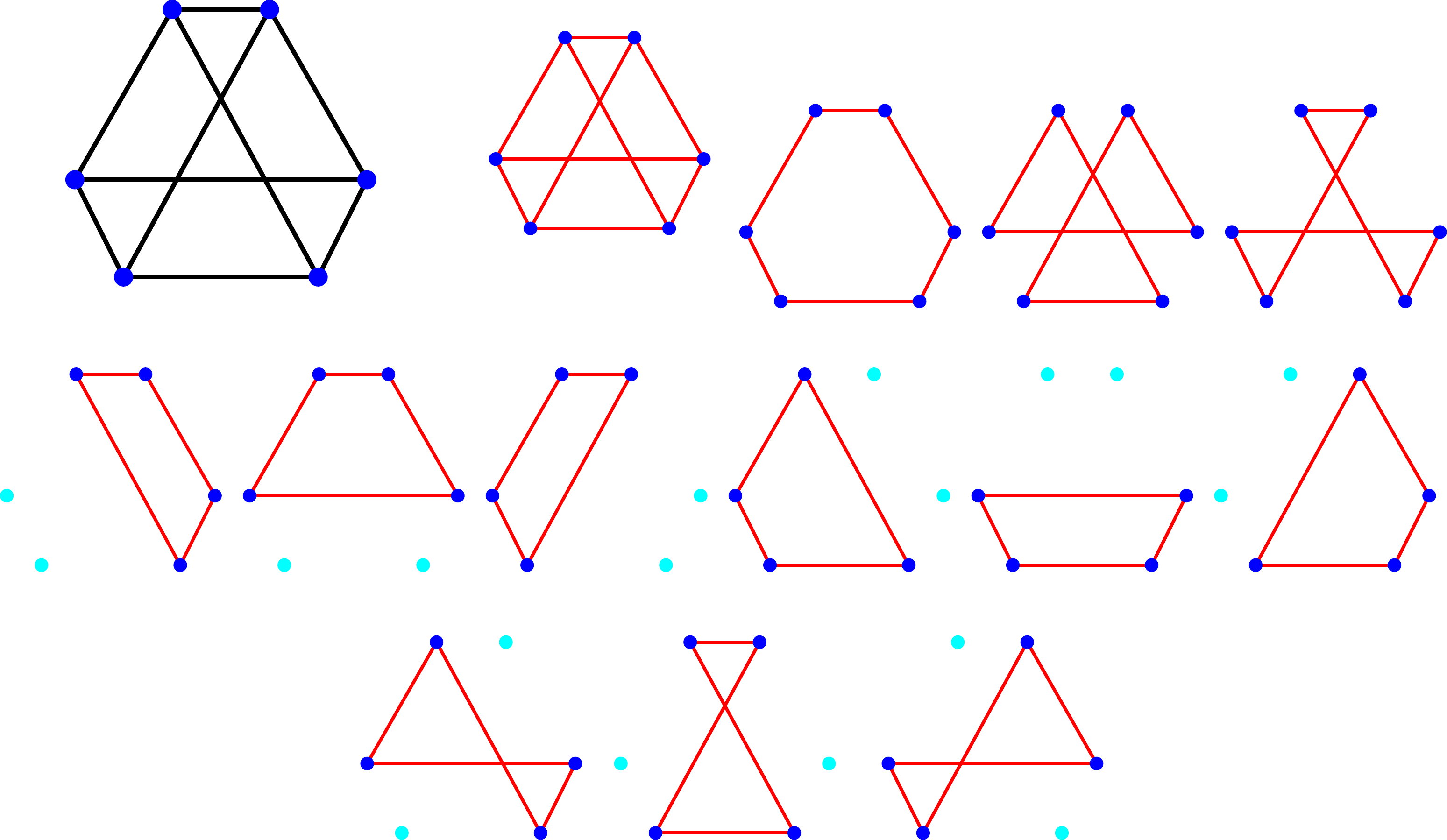}
\end{center}
\caption{GKM-graph $\G_6$ and the list of all its faces of rank $2$. The upper row shows totally geodesic subgraphs. Only the first of them corresponds to the actual face of the full flag manifold $\Fl_3$.}\label{figFlags}
\end{figure}

The aim of this subsection is to prove Theorem~\ref{thmCore}. Let us recall the required definitions. A map $f\colon S\to T$ between two posets is called monotonic if $s_1\leqslant s_2$ in $S$ implies $f(s_1)\leqslant f(s_2)$ in $T$. All maps of posets are assumed monotonic in the following. If there are two maps $f\colon S\to T$ and $g\colon T\to S$, we simply write them as $f\colon S\rightleftarrows T\colon g$.

\begin{defin}\label{defGaloisCon}
A pair of maps $f\colon S\rightleftarrows T\colon g$ is called a Galois connection (an order preserving Galois connection), if for any $s\in S$ one has $s\leqslant g(f(s))$, and for any $t\in T$ there holds $f(g(t))\leqslant t$. If one of the maps is an inclusion of a subposet, this map is called Galois insertion.
\end{defin}

To prove Theorem~\ref{thmCore} we actually prove a more general statement, where the condition of GKM-manifold is replaced by that of normal weak GKM-manifold.

\begin{prop}\label{propGaloisGeneral}
Let $X$ be a normal weak GKM-manifold. Then there exist Galois insertions
\[
S(X)\hookrightarrow S(\G(X)) \mbox{ and } S(X)\hookrightarrow \Stg(\G(X)).
\]
\end{prop}

\begin{proof}
Let us prove the statement for the case $S(\G(X))$, another poset $\Stg(\G(X))$ is completely similar. The inclusion map $\iota\colon S(X)\hookrightarrow S(\G(X))$ is constructed in an obvious manner since every face gives rise to its GKM-graph, which is a face of $\G(X)$ (the normality is needed to assure this graph is connected, required in the definition of GKM-manifold). We now construct the adjoint map $\pi\colon S(\G(X))\to S(X)$. Let $\Hh\subset\G$ be a face of rank $r$. This means by definition that the flat $\Pi_\Hh$ (which is independent of the vertex $x$ of $\Hh$, see~\eqref{eqFlatOfGraph}) has dimension $r$. Consider the subgroup $H\subset T$ which is the connected component of $1$ in $\Ker(\prod_{e\in\str(x)}\alpha\colon T\to \prod_{e\in\str(x)}S^1)$. Let us take the connected component of $X^H$ which contains any $x\in \Hh$. This is a face submanifold $X_F$ (see Lemma~\ref{lemFaceNearFixedPoint}). We set $\pi(\Hh)=X_F$. It is easily seen that $\iota(\pi(\Hh))\supseteq \Hh$ for any GKM-subgraph $\Hh$, and $\pi(\iota(X_F))=X_F$ for any face submanifold. Hence we get a Galois insertion.
\end{proof}

\begin{rem}
Proposition~\ref{propGaloisGeneral} essentially shows that among all GKM-subgraphs $\Hh$ with the given flat $\Pi_\Hh$ at $x$ there exists the greatest GKM-subgraph, corresponding to the actual face of a GKM-manifold (or normal weak GKM-manifold). This observation makes it algorithmically possible to restore the poset of actual faces of $X$, given its GKM-graph with axial function. The algorithm proceeds as follows.
\end{rem}

\begin{algor}\hfill
\begin{enumerate}
  \item List all elements $\Hh$ of $S(\G(X))$ (or $\Stg(\G(X))$), mark each element with the flat $\Pi_\Hh=\langle\alpha(e)\mid e\in\str(x)\rangle$ (for any vertex $x$ of $\Hh$).
  \item In the set of all subgraphs containing vertex $x$ and marked with the given flat $\Pi$, find the greatest-by-inclusion subgraph. Delete everything from this set except the chosen subgraph.
  \item The poset of the remaining subgraphs is isomorphic to $S(X)$.
\end{enumerate}
\end{algor}

\begin{rem}
In alternative terms, one may speak of $S(X)$ as the ``ranked core'' of both posets $S(\G(X))$ and $\Stg(\G(X))$. The relations between GKM-manifolds and the cores of finite topologies will be covered elsewhere.
\end{rem}

To conclude the paper we give a final remark serving as an invitation for future work.

\begin{rem}
It is classical in toric topology to consider torus actions of complexity zero (torus manifolds, quasitoric manifolds) and related constructions such as moment-angle manifolds, see~\cite{BPnew}. In all these cases, the underlying combinatorial objects are given by simplicial complexes, and more generally simplicial posets, in other words, the posets locally modelled by Boolean lattices. In this paper we have shown that geometrical lattices of flats in matroids serve as the building blocks for combinatorial constructions which encode more general torus actions. In the situation when a torus acts on (stably almost) complex manifolds respecting the complex structure, tangent weights of an action are determined without sign ambiguity as vectors in the rational space $\Hom(T,S^1)\otimes\Qo$. We expect that the theory of \emph{oriented matroids} should be a useful tool in the study of such actions. In particular, we suppose that the signs of fixed points --- playing an important role in the study of quasitoric manifolds, --- for general torus actions should be replaced by the ``dual oriented matroids at fixed points''. Here we want to mention the current work of Kuroki et al.~\cite{KurGale} who proceed in this direction by introducing Gale duality in toric topology.
\end{rem}

\end{document}